\documentclass[12pt,english]{amsart}
\usepackage{amsmath,amsthm,amssymb,amsfonts,amscd, array,latexsym, pb-diagram}
\usepackage[noadjust]{cite}
\usepackage[T2A]{fontenc}
\usepackage[cp1251]{inputenc}
\usepackage[english]{babel}
\usepackage{diagmac2} 
\usepackage{tikz}

\usepackage{color} 

\newtheorem{theorem}{Theorem}[section]
\newtheorem{corollary}[theorem]{Corollary}
\newtheorem{lemma}[theorem]{Lemma}
\newtheorem{proposition}[theorem]{Proposition}

\theoremstyle{definition}
\newtheorem{definition}[theorem]{Definition}
\newtheorem{example}[theorem]{Example}

\theoremstyle{remark}
\newtheorem{remark}[theorem]{Remark}
\numberwithin{equation}{section}

\newcommand{\RR}{\mathbb{R}}

\def\<#1>{\langle #1 \rangle}

\newbox\onebox
\newcommand{\coherent}[1]{\mathbin{\setbox\onebox=\hbox{$=$}\lower0.7\ht%
\onebox\hbox{$\stackrel{#1}{=}$}}}

\newcommand{\acr}{\newline\indent}

\begin{document}

\title{Delhomme--Laflamme--Pouzet--Sauer space \\ as groupoid}

\author{Oleksiy Dovgoshey and Alexander Kostikov}

\address{\textbf{Oleksiy Dovgoshey}\acr
Department of Theory of Functions \acr
Institute of Applied Mathematics and Mechanics of NASU \acr
Slovyansk, Ukraine, \acr
Department of Mathematics and Statistics \acr
University of Turku,  Turku, Finland}
\email{oleksiy.dovgoshey@gmail.com, oleksiy.dovgoshey@utu.fi}

\address{\textbf{Alexander Kostikov}\acr
LLC Technical  University
``Metinvest Politechnic'' \acr
Zaporizhia, Ukraine}
\email{alexkst63@gmail.com}

\begin{abstract}
 Let $\mathbb{R}^{+}=[0, \infty)$ and let $d^+$ be the ultrametric on $\RR^+$ such that $d^+ (x,y) = \max\{x,y\}$ for all different $x,y \in \RR^+$.
 It is shown that the monomorphisms of the groupoid $(\RR^+, d^+)$ coincide with the injective  ultrametric-preserving functions and that the auto\-mor\-phisms of $(\RR^+, d^+)$ coincide with the self-homeomorphisms of $\RR^+$.
 The structure of endomorphisms of $(\RR^+, d^+)$ is also described.
\end{abstract}



\maketitle

\section{Introduction. Ultrametrics and pseudoultrametrics}

Let  denote by $\RR^+$ the set of all nonnegative real numbers.

\begin{definition}\label{def1.1}
Let $X$ be a nonempty set. An \emph{ultrametric} on $X$ is a function $d : X \times X \to \RR^{+}$ such that for all $x, y, z \in X$:
\begin{itemize}
\item[(i)] $(d(x, y) = 0) \Longleftrightarrow (x = y)$, the \emph{positivity property};
\item[(ii)] $d(x, y) = d(y, x)$, the \emph{symmetry property};
\item[(iii)] $d(x, y) \leqslant \max\{ d(x, z),  d(z, y)\}$, the \emph{ strong triangle inequality}.
\end{itemize}
\end{definition}

 If $d$ is an ultrametric on $X$, then we will say that $(X,d)$ is an \emph{ultrametric space}.

 The following ultrametric space was introduced by Delhomme, Laflamme, Pouzet and Sauer in \cite{DLPS2008TaiA} and this space is the main focus of our research.

Let us define a mapping $d^+: \RR^{+} \times \RR^+ \to \RR^+$ as
\begin{equation}\label{ex1.2_eq1}
            d^+(p,q) := \left\{
  \begin{array}{ll}
    0,  \quad & \hbox{if} \quad p=q, \\
    \max\{p,q\}, \quad & \hbox{otherwise.}
  \end{array}
\right.
\end{equation}
Then $d^+$ is an ultrametric on $\RR^+$.

Different properties of  ultrametric spaces have been studied in \cite{BS2017, DLPS2008TaiA, DD2010, DM2008, DM2009, DP2013SM, GH1961S, GroPAMS1956, Ibragimov2012, KS2012, Lemin1984FAA, Lemin1984RMS39:5, Lemin1984RMS39:1, Lemin1985SMD32:3, Lemin1988, Lemin2003, PTAbAppAn2014, Qiu2009pNUAA, Qiu2014pNUAA, Vau1999TP, Ves1994UMJ, Ish2021ANUAA, Ish2023TA, Ish2023ANUAA, Ish2024TA, SP2024FPTASE, SKPMS2020, DovBBMSSS2020, DK2023JMS}. The importance of the space $(\RR^+, d^+)$ and its subspaces in the theory of ultrametric spaces was noted by  Yoshito Ishiki in \cite{Ish2023ANUAA}.

The useful generalization of the concept of ultrametric is the concept of pseudoultrametric.

\begin{definition}\label{def2.1}
Let $X$ be a nonempty set and let $d : X \times X \to \RR^{+}$ be a symmetric function such that $d(x, x) = 0$ holds for every $x \in X$. The function $d$ is a \emph{pseudoultrametric} on $X$ if it satisfies the strong triangle inequality.
\end{definition}

If $d$ is a pseudoultrametric on $X$, then $(X, d)$ is called a pseudoultrametric space.

\begin{definition}\label{def2.8}
A function $f : \RR^+ \to \RR^+$ is \emph{ultrametric-preserving (pseudoultrametric-preserving)} if $(X, f \circ d)$ is an ultrametric space (pseudoultrametric space) for every ultrametric (pseudoultrametric) space  $(X, d)$.
\end{definition}

\begin{remark}\label{rem1.2}
Here  we write $f \circ d$ for the mapping
$$
X \times X \overset{d}{\to} \RR^+ \overset{f}{\to} \RR^+ .
$$
\end{remark}

As in \cite{Dov2024arx} we denote  by $\mathbf{P_{U}}$ and the set of all ultrametric-preserving functions and, respectively, by $\mathbf{P_{PU}}$ the set of all pseudoultrametric-preserving ones. We also will use the following designations:

$\mathbf{In} (\mathbf{P_{U}}) $ -- the set  of all injective $f \in \mathbf{P_U}$;

$\mathbf{In} (\mathbf{P_{PU}}) $ -- the set  of all injective $f \in \mathbf{P_{PU}}$;

$\mathbf{Ai} (\mathbf{P_{PU}}) $ -- the set  of all $f \in \mathbf{P_{PU}}$ with injective restriction $f_{|\RR^+ \setminus f^{-1} (0)}$;

$\mathbf{ASI}$ -- the set  of all $f : \RR^+ \to \RR^+$ with $f(0)=0$ and strictly increasing restriction  $f_{|\RR^+ \setminus f^{-1} (0)}$;

$\mathbf{SI} $ -- the set  of all strictly increasing $f \in \mathbf{ASI}$;

$\mathbf{End} (S, \ast ) $  -- the set  of all endomorphisms of a groupoid  $(S, \ast)$;

$\mathbf{Mon} (S, \ast ) $ -- the set  of all monomorphisms of $(S, \ast )$;

$\mathbf{Aut} (S, \ast ) $ -- the set  of all automorphisms of $(S, \ast )$.

The main goal of the present paper is to prove the equalities
\begin{equation}\label{s1_eq3}
\mathbf{ASI} = \mathbf{End} (\RR^+, d^+),
\end{equation}
\begin{equation}\label{s1_eq2}
\mathbf{SI}   = \mathbf{Mon} (\RR^+, d^+).
\end{equation}

The paper is organized as follows. The next section contains some definitions and results related to groupoids, ultrametric-preserving  functions, and pseudoultrametric-preserving ones.

The main results of the paper are given in Section~3. Equalities \eqref{s1_eq3} and \eqref{s1_eq2} are proved in Theorems~\ref{th3.15} and \ref{th3.4} respectively.

Theorem~\ref{th3.9} shows that $\mathbf{Aut} (\RR^+, d^+)$ coincides with the set of all  self-homeomorphisms of $\RR^+$.


 \section{Preliminaries on ultrametric-preserving  functions  and groupoids}


Recall that $f : \RR^{+} \to \RR^+$ is \emph{increasing} iff
$$
(a \geqslant b) \Longrightarrow (f (a) \geqslant f (b))
$$
holds for all $a, b \in \RR^+$. Moreover, $f : \RR^+ \to \RR^+$ is \emph{strictly increasing} iff
$$
(a> b)  \Longrightarrow (f(a) > f(b))
$$
holds for all distinct $a, b \in \RR^+$.

A function $f : \RR^{+} \to \RR^+$ is said to be \emph{amenable} if $f^{-1}(0)=\{0\}$.

P.~Ponsgriiam and I. Termwuttipong found the following simple characterization of $\mathbf{P_U}$ in \cite{PTAbAppAn2014}.

\begin{theorem}\label{th2.9}
A function $f : \RR^+ \to \RR^+$ is ultrametric-preserving if and only if $f$ is amenable and increasing.
\end{theorem}

The next extension of Theorem~\ref{th2.9} was obtained in \cite{Dov2020MS}.

\begin{proposition}\label{pr2.4}
The following conditions are equivalent for every function $f : \RR^+ \to \RR^+$.
\begin{itemize}
\item[(i)] $f$ is increasing and $f (0) = 0$ holds.
\item[(ii)] $f$ is pseudoultrametric-preserving.
\end{itemize}
\end{proposition}

\begin{remark}\label{rem2.3}
The set $\mathbf{P_U}$ of ultrametric-preserving functions was also studied in \cite{Dov24TA, BDa2024, BDS2021, KPS2019IJoMaCS} and \cite{VD2021MS}.
\end{remark}

Proposition~\ref{pr2.4} and the definition of $\mathbf{SI}$ imply the following corollary.

\begin{corollary}\label{cor2.3}
The equality
\begin{equation}\label{cor2.3_eq1}
\mathbf{SI}= \mathbf{In} (\mathbf{P_{PU}})
\end{equation}
holds.
\end{corollary}

\begin{proof}
To prove \eqref{cor2.3_eq1} it is sufficient to note that an increasing function $f$ is  strictly increasing if and only if  $f$ is injective.
\end{proof}

\begin{lemma}\label{lem2.3}
Let us consider arbitrary  $f \in \mathbf{In}(\mathbf{P_{PU}})$. Then $f$ is ultrametric-preserving,
\begin{equation}\label{lem2.3_eq1}
    f \in \mathbf{P_U}.
\end{equation}
\end{lemma}

\begin{proof}
It follows from $f \in \mathbf{In} (\mathbf{P_{PU}})$ that $f \in \mathbf{P_{PU}}$. Since $f$  belongs to $\mathbf{P_{PU}}$, $f$ is increasing and the equality
\begin{equation}\label{lem2.3_preq1}
f(0)=0
\end{equation}
holds by Proposition~\ref{pr2.4}. The injectivity of $f$ and \eqref{lem2.3_preq1} imply
$$
f^{-1} (0) = \{0\}.
$$
Thus $f$ is  amenable and increasing. Hence \eqref{lem2.3_eq1} holds by Theorem~\ref{th2.9}
\end{proof}

Lemma~\ref{lem2.3} gives us the next proposition wich will be used in Section~3 below.

\begin{proposition}\label{pr2.5}
The equality
\begin{equation}\label{pr2.5_eq1}
  \mathbf{In} (\mathbf{P_{PU}}) = \mathbf{In} (\mathbf{P_U})
\end{equation}
holds.
\end{proposition}

\begin{proof}

Theorem~\ref{th2.9} and Proposition~\ref{pr2.4} imply the inclusion
\begin{equation}\label{pr2.5_preq2}
\mathbf{In} ( \mathbf{P_{PU}}) \supseteq \mathbf{In} (\mathbf{P_{U}}).
\end{equation}
Suppose now that $f : \RR^+ \to \RR^+$ is an arbitrary element of the set $\mathbf{In} (\mathbf{P_{PU}})$. Then the membership
\begin{equation}\label{pr2.5_preq3}
f \in \mathbf{P_U}
\end{equation}
is valid by Lemma~\ref{lem2.3}. Since $f$ is injective, \eqref{pr2.5_preq3} implies
$$
f \in \mathbf{In} (\mathbf{P_U}).
$$
Consequently the inclusion
$$
\mathbf{In} (\mathbf{P_{PU}}) \subseteq \mathbf{In} (\mathbf{P_{U}})
$$
holds. The last inclusion and \eqref{pr2.5_preq2} give us \eqref{pr2.5_eq1}.
\end{proof}

Let us recall some definitions connected with the concept of groupoid.

\begin{definition}\label{def2.4}
A \emph{groupoid} is a pair $(S, \ast)$ consisting of a set $S$ and a binary operation $\ast : S \times S \to S$ which is called the \emph{composition} on $S$. An element $e \in S$ of a groupoid $(X, \ast)$ is said to be the \emph{identity} of $S$  if the equalities
$$e \ast s = s \ast e = s$$
hold for every $s \in S$.
\end{definition}

\begin{remark}\label{rem2.7}
Nicolas Bourbaki uses the term ``unital magma'' to refer to a groupoid with identity element (see Definition~2 in \cite[p.~12]{Bourbaki1974}).
\end{remark}

It is easy to prove that, for arbitrary groupoid $(S, \ast )$, the identity element is unique if it exists, see, for example \cite{R1968}, p.~111, Proposition~1. In what follows we denote such element as $1_S$.

Let us consider the basic example of a groupoid  for us.

\begin{example}\label{ex2.11}
Let $(\RR^+, d^+)$ be the ultrametric space defined by formula~\eqref{ex1.2_eq1}.   Then $(\RR^+, d^+)$ is a groupoid with the composition $d^+$ and  the identity element
\begin{equation}\label{ex2.3_eq2}
1_{\RR^+}=0.
\end{equation}
To see that \eqref{ex2.3_eq2} holds it suffices to note that the equalities
$$
 d^+ (x,0) = x = d^+ (0,x)
$$
follow from \eqref{ex1.2_eq1} for each $x \in \RR^+$.
\end{example}

\begin{definition}\label{def2.5}
A groupoid $(S, \ast)$ is said to be a \emph{monoid} if $S$ contains an identity element and $\ast$   is associative.
\end{definition}

\begin{example}\label{ex2.6}
Let us define a composition $\vee$ on the set $\RR^+$ as
\begin{equation*}\label{ex2.6_eq1}
x \vee y  = \max \{x,y\}
\end{equation*}
for all $x,y \in \RR^+$. Then $(\RR^+, \vee)$ is a monoid with the identity element
\begin{equation}\label{ex2.6_eq2}
1_{\RR^+}=0.
\end{equation}
To see that $(\RR^+,\vee)$ really is a monoid and \eqref{ex2.6_eq2} holds, it suffices to note that the equalities
$$
\max \{ \{x,y\}, z\} = \max \{x,y,z\} = \max \{x, \{y,z\}\}
$$
and
$$
\max \{x,0\} =  \max \{ 0, x\}= x
$$
are satisfied for all $x,y,z \in \RR^+$.
\end{example}

\begin{proposition}\label{prop2.10}
The groupoid $(\RR^+, d^+)$  is not a monoid.
\end{proposition}

\begin{proof}
Indeed, let $x, y \in \RR^+$ and let the double inequality
\begin{equation}\label{prop2.10_eq1}
x>y>0
\end{equation}
hold. Then using \eqref{ex1.2_eq1} and \eqref{prop2.10_eq1} we obtain
$$
d^+ (d^+ (x, x), y) = d^+(0,y) = y \neq  0 = d^+(x,x) = d^+(x, d^+(x,y)).
$$
Thus the composition $d^+$  is not associative.
\end{proof}

\begin{definition}\label{def2.5*}
Let $S=(S, \ast, 1_S)$  be a groupoid. Then a mapping $\Phi: S \to S$ is called an \emph{endomorphism} if, for all $x,y \in S$, we have
$$
\Phi (x \ast y) = \Phi (x) \ast \Phi (y)
$$
and the equality
$$
\Phi (1_S) = 1_S
$$
holds. An injective endomorphism $S \to S$ is called a \emph{monomorphism} of~$S$. The bijective endomorphisms of $S$ are called  the \emph{authomorphisms} of~$S$.
\end{definition}

\begin{lemma}\label{lem3.5}
A function $ f:  \RR^+ \to \RR^+$ is an endomorphism of the groupoid $(\RR^+, d^+)$ if and only if $f(0) = 0 $ and the equality
\begin{equation*}\label{lem3.5_eq1}
 f(d^+ (x,y)) =  d^+ (f(x), f(y))
\end{equation*}
holds for all $x,y \in \mathbb{R}^+$.
\end{lemma}

\begin{proof}
It follows from Definitions~\ref{def2.4} and \ref{def2.5*}.
\end{proof}

\section{Main results}

Let us turn now to the set $\mathbf{End} (\RR^+, d^+)$  of endomorphisms of the groupoid $(\RR^+, d^+)$.

\begin{lemma}\label{pr3.10}
The inclusion
\begin{equation}\label{pr3.10_eq1}
    \mathbf{End} (\RR^+, d^+)  \subseteq  \mathbf{P_{PU}}
\end{equation}
holds.
\end{lemma}

\begin{proof}
Let us consider an arbitrary endomorphism $f$ of the groupoid $(\RR^+, d^+)$,
\begin{equation}\label{pr3.10_preq1}
    f \in \mathbf{End} (\RR^+, d^+).
\end{equation}
To prove inclusion \eqref{pr3.10_eq1} we must show that
\begin{equation}\label{pr3.10_preq2}
    f \in  \mathbf{P_{PU}}.
\end{equation}
By Proposition~\ref{pr2.4} membership~\eqref{pr3.10_preq2} holds if $f$ is increasing and satisfies the equality
\begin{equation}\label{pr3.10_preq3}
    f(0)=0.
\end{equation}
Equality~\eqref{pr3.10_preq3} follows from Definition~\ref{def2.5*} and membership~\eqref{pr3.10_preq1}. Thus, it is enough to prove that $f$ is an increasing function, i.e.
\begin{equation}\label{pr3.10_preq4}
   f(x) \leqslant f(y)
\end{equation}
holds whenever  $x,y \in \mathbb{R}^+$ and
\begin{equation}\label{pr3.10_preq5}
    x \leqslant y.
\end{equation}
Let us consider arbitrary $x, y \in \RR^+$ satysfying \eqref{pr3.10_preq5}. By Lemma~\ref{lem3.5} membership~\eqref{pr3.10_preq1} implies the equality
\begin{equation}\label{pr3.10_preq6}
     f(d^+(x,y)) = d^+ (f(x), f(y)).
\end{equation}
Using~\eqref{ex1.2_eq1} and \eqref{pr3.10_preq5} we obtain the equality $d^+(x,y) = y$. The last equality and \eqref{pr3.10_preq6} give us
\begin{equation}\label{pr3.10_preq7}
  f(y)= d^+(f(x), f(y)).
\end{equation}
Suppose first that $f(y) =0$. Then we can rewrite \eqref{pr3.10_preq7} us
\begin{equation}\label{pr3.10_preq8}
     d^+(f(x), f(y)) =0.
\end{equation}
Since $d^+$ is an  ultrametric, \eqref{pr3.10_preq8} is valid iff
$$
f(x) = f(y),
$$
that implies \eqref{pr3.10_preq4}.
If $f(y) \neq 0$, then \eqref{pr3.10_preq7} implies the equality
\begin{equation}\label{pr3.10_preq9}
    f(y) = \max \{f(x), f(y)\}.
\end{equation}
Inequality~\eqref{pr3.10_preq4} follows from~\eqref{ex1.2_eq1} and \eqref{pr3.10_preq9}.

The proof is completed.
\end{proof}

The following example shows that  $\mathbf{End}  (\RR^+, d^+)$ is a proper subset of $\mathbf{P_{PU}}$.

\begin{example}\label{ex3.11}
Let $a$ be an arbitrary point of $(0,  \infty)$. Let us define $f: \RR^+ \to \RR^+$ as
$$
f(t) = \left\{
  \begin{array}{ll}
   0, & \quad \hbox{if} \quad t=0, \\
   a, & \quad \hbox{otherwise.}
  \end{array}
\right.
$$
Proposition~\ref{pr2.4} implies $f \in \mathbf{P_{PU}}$. Suppose that $f$ is an endomorphism of $(\RR^+, d^+)$. Then, by Lemma~\ref{lem3.5}, the equality
\begin{equation}\label{ex3.11_eq1}
    f(d^+(x,y)) = d^+ (f(x), f(y))
\end{equation}
holds for all $x,y \in \RR^+$.  In particular, for $x=0$ and $y=1$, using \eqref{ex1.2_eq1} we obtain
\begin{equation}\label{ex3.11_eq2}
   f(a)= f(d^+(0,1)) = d^+ (f(0), f(1))= d^+ (0,a) = a.
\end{equation}
Similarly if $x=1$ and $y=2$ then~\eqref{ex3.11_eq1} implies
\begin{equation}\label{ex3.11_eq3}
   f(a)= f^+(d^+(1,2)) = d^+ (f(1), f(2)) = d^+ (a,a) = 0.
\end{equation}
It follows from~\eqref{ex3.11_eq2} and \eqref{ex3.11_eq3} that $a=0$, contrary to $a>0$. Thus we have
$$
f \in \mathbf{P_{PU}}  \quad \textrm{and} \quad f \notin \mathbf{End} (\RR^+, d^+).
$$
\end{example}

By analysing Example~\ref{ex3.11}, we can prove the following lemma.

\begin{lemma}\label{lem3.12}
    Let $a  \in (0, \infty)$ and let $f: \RR^+ \to \RR^+$ be a function such that $f(0)=0$ and
    $$
      f(x) = f(y) = a
    $$
    for some different $x,y \in (0, \infty)$. Then $f$ is not an endomorphism of the groupoid $(\RR^+,d^+)$,
    $$
       f \notin \mathbf{End} (\RR^+, d^+).
    $$
\end{lemma}

\begin{lemma} \label{lem3.13}
The inclusion
\begin{equation}\label{lem3.13_eq1}
\mathbf{End} (\RR^+, d^+) \subseteq \mathbf{ASI}
\end{equation}
holds.
\end{lemma}

\begin{proof}
Let us consider an arbitrary function
\begin{equation}\label{lem3.13_preq1}
f \in \mathbf{End} (\RR^+, d^+).
\end{equation}
We must show that
\begin{equation}\label{lem3.13_preq2}
f \in \mathbf{ASI}.
\end{equation}
Suppose contrary that
\begin{equation}\label{lem3.13_preq3}
f \notin \mathbf{ASI}.
\end{equation}
By Lemma~\ref{pr3.10}, membership~\eqref{lem3.13_preq1} implies
\begin{equation}\label{lem3.13_preq4}
f \in \mathbf{P_{PU}}.
\end{equation}
Now using~\eqref{lem3.13_preq4} and Proposition~\ref{pr2.4} we obtain that $f$  is increasing and satisfies $f(0)=0$. Hence~\eqref{lem3.13_preq3} implies that the restriction of $f$ on the set $\RR^+ \setminus f^{-1} (0)$ is not strictly increasing. An increasing function is strictly increasing iff it is an injective function. Thus $f_{|\RR^+\setminus f^{-1} (0)}$ is not injective. Consequently there are $a \in (0, \infty)$ and $x,y \in \RR^+ \setminus f^{-1} (0)$ such that
\begin{equation}\label{lem3.13_preq5}
f (x) = f(y) = a.
\end{equation}
The equality $f(0)=0$ implies
$$
\RR^+ \setminus f^{-1} (0) \subseteq (0, \infty).
$$
Consequently the points $a, x $ and $y$  belong to $(0, \infty)$. Now applying Lemma~\ref{lem3.12} we obtain
$$
f \notin \mathbf{End} (\RR^+, d^+)
$$
contrary to \eqref{lem3.13_preq1}. Thus \eqref{lem3.13_preq2} is valid, that implies ~\eqref{lem3.13_eq1}. The proof is completed.
\end{proof}

\begin{lemma}\label{lem3.14}
The equality
$$
\mathbf{ASI} = \mathbf{Ai} (\mathbf{P_{PU}})
$$
holds.
\end{lemma}

\begin{proof}
Let $f: \RR^+ \to \RR^+$  be an increasing function such that $f(0)=0$. Then the restriction $f_{|\RR^+ \setminus f^{-1} (0)}$ is strictly increasing iff this restriction is injective.
\end{proof}

The following theorem is the first main result of the paper.

\begin{theorem}\label{th3.15}
The equalities
\begin{equation}\label{th3.15_eq1}
\mathbf{End} (\RR^+, d^+) = \mathbf{ASI} = \mathbf{Ai} (\mathbf{P_{PU}})
\end{equation}
holds.
\end{theorem}

\begin{proof}
By Lemma~\ref{lem3.14} we have the equality
$$
\mathbf{ASI} = \mathbf{Ai} (\mathbf{P_{PU}}).
$$
Moreover, Lemma~\ref{lem3.13} gives us the inclusion
$$
\mathbf{End} (\RR^+, d^+) \subseteq \mathbf{ASI}.
$$
Hence \eqref{th3.15_eq1} holds iff
\begin{equation}\label{th3.15_preq1}
\mathbf{ASI} \subseteq \mathbf{End} (\RR^+, d^+).
\end{equation}
To prove \eqref{th3.15_preq1} let us consider an arbitrary function $f \in \mathbf{ASI}$. It is sufficient to show that
\begin{equation}\label{th3.15_preq2}
f \in \mathbf{End} (\RR^+, d^+).
\end{equation}
By Lemma~\ref{lem3.5} membership~\eqref{th3.15_preq2} is valid if
\begin{equation}\label{th3.15_preq3}
  f(d^+(x,y))  = d^+ (f(x), f(y))
\end{equation}
holds  for all $x,y  \in \RR^+$.

Let us consider arbitrary $x,y \in\RR^+$. Suppose first that
\begin{equation}\label{th3.15_preq4}
x,y \in f^{-1} (0).
\end{equation}
Assume, without loss of generality, that
$$
x \geqslant y.
$$
Then $d^+ (x,y) = x$ holds by \eqref{ex1.2_eq1} and we have
\begin{equation}\label{th3.15_preq5}
 f(x) = f(y) = 0
\end{equation}
by \eqref{th3.15_preq4}.

Now using \eqref{th3.15_preq4} and \eqref{th3.15_preq5} we obtain \eqref{th3.15_preq2},
$$
f(d^+ (x,y))  = f(x) = 0 = d^+(0,0) = d^+(f(x), f(y)).
$$
Suppose now that exactly one from the points $x,y$ belongs to $f^{-1} (0)$. WLOG , let $x \in \RR^+ \setminus f^{-1} (0)$ and $y \in f^{-1} (0)$. Then we have
\begin{equation}\label{th3.15_preq6}
 f(x) > 0  \quad \textrm{and} \quad f^{-1}(y) =0.
\end{equation}
The membership $f \in \mathbf{ASI}$ implies that $f$ is increasing. Consequently~\eqref{th3.15_preq6} implies the inequality
\begin{equation}\label{th3.15_preq7}
d^+ (f(x),f(y))  = f(x).
\end{equation}
Now using \eqref{th3.15_preq6}, \eqref{th3.15_preq7} and \eqref{ex1.2_eq1} we obtain
\begin{equation}\label{th3.15_preq8}
f(d^+(x,y)) = f(x)
\end{equation}
and
\begin{equation}\label{th3.15_preq9}
d^+(f^+(x), f^+(y)) = f(y).
\end{equation}
Thus \eqref{th3.15_preq3} holds. To complete the proof of \eqref{th3.15_preq2}, it remains to consider the case when
\begin{equation}\label{th3.15_preq10}
x,y \in\RR^+ \setminus f^{-1} (0).
\end{equation}
If $x=y$, then $f(x) = f(y)$ and consequently  we have
$$
f(d^+(x,y)) = f(0) =0 = d^+ (f(x), f(y)),
$$
that implies  \eqref{th3.15_preq2}. Suppose that $x \neq y$. WLOG let $x>0$.

Then \eqref{th3.15_preq10} implies
$$
f(x) > f(y).
$$
Therefore \eqref{th3.15_preq8} and \eqref{th3.15_preq9} follows as above.

Thus equality~\eqref{th3.15_preq3} is valid for all $x,y \in \RR^+$, that implies the validity of \eqref{th3.15_preq2}.

The proof is completed.
\end{proof}

The next theorem  was proved in \cite{Dov2024arx}.

\begin{theorem}\label{mainth}
The sets $\mathbf{End} (\RR^+, \vee)$ and $\mathbf{P}_{\mathbf{PU}}$ are the same,
\begin{equation}\label{mth_eq1}
    \mathbf{End} ({\RR^+}, \vee)= \mathbf{P}_{\mathbf{PU}}.
\end{equation}
\end{theorem}

This theorem implies the following

\begin{lemma}\label{lem3.2}
The equality
\begin{equation}\label{lem3.2_eq1}
   \mathbf{Mon} ({\RR^+}, \vee) = \mathbf{In} (\mathbf{P}_{\mathbf{PU}})
\end{equation}
holds.
\end{lemma}

\begin{proof}
Equality \eqref{lem3.2_eq1} follows from \eqref{mth_eq1} and the definitions of $\mathbf{Mon} (\RR^+, \vee)$ and $\mathbf{In} (\mathbf{P_{PU}})$.
\end{proof}

The next theorem is the second main result of the paper.

\begin{theorem}\label{th3.4}
    The equalities
    \begin{equation}\label{th3.4_eq1}
        \mathbf{Mon} (\RR^+, d^+)  = \mathbf{SI} = \mathbf{In} (\mathbf{P_U}) = \mathbf{In}  (\mathbf{P_{PU}}) = \mathbf{Mon} (\RR^+, \vee)
    \end{equation}
    hold.
\end{theorem}

\begin{proof}
    The equality
    $$
\mathbf{Mon} (\RR^+, \vee) = \mathbf{In} (\mathbf{P_{PU}})
    $$
    was proved in Lemma~\ref{lem3.2}. Lemma~\ref{lem2.3} and Proposition~\ref{pr2.5} give us the equalities
    $$
\mathbf{SI} = \mathbf{In} (\mathbf{P_U}) = \mathbf{In} (\mathbf{P_{PU}}).
    $$
    Consequently \eqref{th3.4_eq1} holds iff
    \begin{equation}\label{th3.4_preq1}
        \mathbf{Mon} (\RR^+, d^+) = \mathbf{SI}.
    \end{equation}

    Theorem~\ref{th3.15} implies the equality
    $$
\mathbf{End} (\RR^+, d^+) = \mathbf{ASI}.
    $$
    The last equality and Definition~\ref{def2.5*} give us
    \begin{equation}\label{th3.4_preq2}
      \mathbf{Mon} (\RR^+, d^+)= \{ f \in \mathbf{ASI} : f \textrm{ is injective}\}.
    \end{equation}
    Now using the definition of the sets $\mathbf{ASI}$ and $\mathbf{SI}$  we see that
    \begin{equation}\label{th3.4_preq3}
      \mathbf{SI} = \{ f \in \mathbf{ASI} : f \textrm{ is injective}\}.
    \end{equation}
    Equality~\eqref{th3.4_preq1} follows from equalities \eqref{th3.4_preq2}--\eqref{th3.4_preq3}.

    The proof is completed.
\end{proof}

\begin{corollary}\label{cor3.5}
A function $f: \RR^+ \to \RR^+$ is a monomorphism of the groupoid $(\RR^+, d^+)$ if and only if $f$ is amenable and strictly increasing.
\end{corollary}

\begin{proof}
It follows from \eqref{th3.4_eq1} and Theorem~\ref{th2.9}.
\end{proof}

Our next goal is to characterize the set $\mathbf{Aut} (\RR^+, d^+)$ of all automorphisms of the groupoid $(\RR^+, d^+)$.

\begin{lemma}\label{lem3.6}
Let $f: \RR^+ \to \RR^+$ be increasing. Write
\begin{equation*}\label{lem3.6_eq1}
    f (x_0+0)  := \inf \{ f(x) : x \in (x_0, \infty) \},
\end{equation*}
and
\begin{equation*}\label{lem3.6_eq2}
    f (x_0 - 0) := \sup \{ f(x) : x \in [0, x_0) \}
\end{equation*}
for each $x_0 \in (0, \infty)$ and, in addition, denote by $f(0+0)$ the infinum of the set $\{ f(x): x \in (0, \infty)\}$,
\begin{equation*}\label{lem3.6_eq3}
    f (0+0) : = \inf \{ f(x) : x \in ( 0, \infty) \}.
\end{equation*}
Then $f$ is a continuous function on $\RR^+$ if and only if
\begin{equation*}\label{lem3.6_eq4}
    f ( 0)  = f (0+0),
\end{equation*}
and the equalities
\begin{equation*}\label{lem3.6_eq5}
f(x_0 - 0)  = f(x_0)  = f(x_0 +0)
\end{equation*}
hold for each $x_0 \in (0, \infty)$.
\end{lemma}

For a proof see \cite[pages 204--205]{Natanson1983}.

The next lemma follows directly from Theorem~5 of \cite{Bourbaki1995} (see pages 338--339).

\begin{lemma}\label{lem3.7}
A mapping $f : \RR^+ \to \RR^+$ is a self-homeomorphism of $\RR^+$ if and only if
$$
f(\RR^+) = \RR^+
$$
and $f$ is strictly monotonic and continuous on $\RR^+$.
\end{lemma}

\begin{remark}\label{rem3.8}
In Lemma~\ref{lem3.7} and Theorem~\ref{th3.9} we consider $\RR^+$ as a topological space with topology included by the usual metric
$$
d(x,y) = |x-y|, \quad \textrm{for} \  x,y \in \RR^+.
$$
\end{remark}

\begin{theorem}\label{th3.9}
The following statements are equivalent for every function $F: \RR^+ \to \RR^+$.
\begin{itemize}
    \item[(i)] $F \in \mathbf{Aut} (\RR^+, d^+)$.

    \item[(ii)] $F$ is strictly increasing  and satisfies  the equality
    \begin{equation}\label{th.3.9_eq1}
     F(\RR^+) = \RR^+.
    \end{equation}

    \item[(iii)] $F$ is a self-homeomorphism of $\RR^+$.
\end{itemize}
\end{theorem}

\begin{proof}
$(i) \Longrightarrow (ii)$. Let $F$ belong  to $\mathbf{Aut} (\RR^+, d^+)$. Then
$$
F \in \mathbf{Mon} (\RR^+, d^+)
$$
holds and, consequently, $F$ is strictly increasing by Corollary~\ref{cor3.5}. The validity of~\eqref{th.3.9_eq1} follows from Definition~\ref{def2.5*}. Thus $(ii)$ holds.

$(ii) \Longrightarrow (iii)$. Let $(ii)$ hold. Since $F$ is strictly increasing, $F$ is strictly monotonic. Moreover, \eqref{th.3.9_eq1} holds by statement $(ii)$. Thus, by Lemma~\ref{lem3.7}, statement $(iii)$ is valid if $F$ is continuous.

Suppose contrary that $F$ is a discontinuous function. Then, by Lemma~\ref{lem3.6},
\begin{equation}\label{th3.9_preq1}
f(0) < f( 0+0)
\end{equation}
or there is $x_0 \in (0, \infty)$ such that
\begin{equation}\label{th3.9_preq1-1}
f(x_0- 0) < f( x_0+0).
\end{equation}
If \eqref{th3.9_preq1} holds, then using \eqref{lem3.6_eq3} we obtain the equality
\begin{equation}\label{th3.9_preq2}
    (f(0),  f(0+0))  \cap f(\RR^+) = \emptyset
\end{equation}
where
\begin{equation}\label{th3.9_preq3}
(f(0), f(0+0)) = \{ x\in \RR^+: f(0) <x  < f(0+0)\}.
\end{equation}
The interval $(f(0), f(0+0))$ is nonempty by \eqref{th3.9_preq1}. Hence $F(\RR^+)$ is a proper subset of $\RR^+$,
\begin{equation}\label{th3.9_preq4}
F(\RR^+) \subsetneq \RR^+
\end{equation}
contrary to \eqref{th.3.9_eq1}.

If inequality \eqref{th3.9_preq1-1} satisfied for some point $x_0 \in (0, \infty)$, then reasoning in a similar way we can prove that \eqref{th3.9_preq4} also holds.
Since~\eqref{th3.9_preq4} contradicts~\eqref{th.3.9_eq1}, statement $(iii)$ follows.

$(iii) \Longrightarrow (i)$. Let $F : \RR^+ \to \RR^+$ be a homeomorphism. We must prove that
\begin{equation}\label{th3.9_preq5}
    F \in \mathbf{Aut} (\RR^+, d^+).
\end{equation}
By Lemma~\ref{lem3.7}, equality~\eqref{th.3.9_eq1} holds. Consequently, by Definition~\ref{def2.5*}, \eqref{th3.9_preq5} is valid iff
\begin{equation}\label{th3.9_preq6}
F \in \mathbf{Mon} (\RR^+, d^+).
\end{equation}
It follows from Corollary~\ref{cor3.5} that \eqref{th3.9_preq6} holds if $F$ is strictly increasing and satisfies the equality
\begin{equation}\label{th3.9_preq7}
F (0)=0.
\end{equation}
Let us prove equality~\eqref{th3.9_preq7}.

The restriction of $F$ on the interval $(0, \infty)$ is a homeomorphism of $(0, \infty)$  on $\RR^+ \setminus \{ F(0)\}$. Consequently $\RR^+  \setminus \{F(0)\}$ is a connected subset of $\RR^+$. The last statement is valid iff \eqref{th3.9_preq7} holds.

Equality~\eqref{th3.9_preq7} and Lemma~\ref{lem3.7} imply that $F$ is strictly increasing. Indeed, let $x_0$ be an arbitrary point of $(0, \infty)$. Since $F$ is a bijective mapping~\eqref{th3.9_preq7} implies that
$$
F(x_0) \in (0, \infty).
$$
Thus the inequality
\begin{equation}\label{th3.9_preq8}
F(x_0) > F(0)
\end{equation}
holds. By Lemma~\ref{lem3.7}, the mapping $F$ is strictly monotonic. The last property and \eqref{th3.9_preq8} imply that $F$ is strictly  increasing.

Thus~\eqref{th3.9_preq6} holds. Statement $(i)$ follows.

The proof is completed.
\end{proof}

It should be noted that the concept of ultrametric-preserving functions has been extended to the special case of ``ultrametric distances'' (see \cite{Dov2019a}). These distances were
introduced by Priess--Crampe and Ribenboim \cite{PR1993AMSUH} and were studies
by different researchers \cite{PR1996AMSUH, PR1997AMSUH, Rib1996PMH, Rib2009JoA}. It seems to be interesting to find generalizations of the main results of the present paper to groupoids generated by such distances.

\section*{Conflict of interest statement}

The research was conducted in the absence of any commercial or financial relationships that could be construed as a potential conflict of interest.

\section*{Funding}

Oleksiy Dovgoshey was supported by grant 359772 of the Academy of Finland.

\end{document}